\documentclass[11pt,leqno]{article}
\usepackage{amsfonts}
\usepackage{amsmath, amsthm, amsfonts, amssymb, color}
\usepackage{mathrsfs}
\renewcommand{\bar}{\overline}
\renewcommand{\hat}{\widehat}
\renewcommand{\tilde}{\widetilde}

\newtheorem{thm}{Theorem}[section]

\newtheorem{lem}[thm]{Lemma}

\newtheorem{exa}[thm]{Example}
\theoremstyle{definition}
\newtheorem{defn}{Definition}[section]
\newcommand{\scr}[1]{\mathscr #1}
\definecolor{wco}{rgb}{0.5,0.2,0.3}

\numberwithin{equation}{section} \theoremstyle{remark}
\newtheorem{rem}{Remark}[section]

\newcommand{\ua}{\uparrow}

\title{{\bf Moderate Deviation and Central Limit Theorem   for SDDEs with Polynomial Growth}
}
\author{
{\bf  Yongqiang Suo,  Jin Tao, Wei Zhang\thanks{zhangyunfei78@csu.edu.cn}}\\
\footnotesize{School of Mathematics and Statistics, Central South
University, Changsha 410083, China}\\
\footnotesize{bellekelly@csu.edu.cn}
}
\begin{document}
\def\A{\mathscr{A}}
\def\G{\mathscr{G}}
\def\eq{\equation}
\def\bg{\begin}
\def\ep{\epsilon}
\def\111{±ßÖµÎÊÌâ(1)--(2)}
\def\x{\|x\|}
\def\y{\|y\|}
\def\xr{\|x\|_r}
\def\xrr{(\sum_{i=1}^T|x_i|^r)^{\frac{1}{r}}}
\def\R{\mathbb R}
\def\ff{\frac}
\def\ss{\sqrt}
\def\B{\mathbf B}
\def\N{\mathbb N}
\def\kk{\kappa} \def\m{{\bf m}}
\def\dd{\delta} \def\DD{\Dd} \def\vv{\varepsilon} \def\rr{\rho}
\def\<{\langle} \def\>{\rangle} \def\GG{\Gamma} \def\gg{\gamma}
  \def\nn{\nabla} \def\pp{\partial} \def\EE{\scr E}
\def\d{\text{\rm{d}}} \def\bb{\beta} \def\aa{\alpha} \def\D{\scr D}
  \def\si{\sigma} \def\ess{\text{\rm{ess}}}\def\lam{\lambda}
\def\beg{\begin} \def\beq{\begin{equation}}  \def\F{\scr F}
\def\Ric{\text{\rm{Ric}}} \def \Hess{\text{\rm{Hess}}}
\def\e{\text{\rm{e}}} \def\ua{\underline a} \def\OO{\Omega}  \def\oo{\omega}
 \def\tt{\tilde} \def\Ric{\text{\rm{Ric}}}
\def\cut{\text{\rm{cut}}} \def\P{\mathbb P} \def\ifn{I_n(f^{\bigotimes n})}
\def\C{\scr C}      \def\alphaa{\mathbf{r}}     \def\r{r}
\def\gap{\text{\rm{gap}}} \def\prr{\pi_{{\bf m},\varrho}}  \def\r{\mathbf r}
\def\Z{\mathbb Z} \def\vrr{\varrho} \def\l{\lambda}
\def\L{\scr L}\def\Tilde{\tilde} \def\TILDE{\tilde}\def\II{\mathbb I}
\def\i{{\rm in}}\def\Sect{{\rm Sect}}\def\E{\mathbb E} \def\H{\mathbb H}
\def\M{\scr M}\def\Q{\mathbb Q} \def\texto{\text{o}} \def\LL{\Lambda}
\def\Rank{{\rm Rank}} \def\B{\scr B} \def\i{{\rm i}} \def\HR{Hat{\R}^d}
\def\to{\rightarrow}\def\l{\ell}\def\ll{\lambda}
\def\8{\infty}\def\ee{\epsilon} \def\Y{\mathbb{Y}} \def\lf{\lfloor}
\def\rf{\rfloor}\def\3{\triangle}\def\H{\mathbb{H}}\def\S{\mathbb{S}}\def\1{\lesssim}
\def\va{\varphi}

\def\R{\mathbb R}  \def\ff{\frac} \def\ss{\sqrt} \def\B{\mathbf
B}
\def\N{\mathbb N} \def\kk{\kappa} \def\m{{\bf m}}
\def\dd{\delta} \def\DD{\Delta} \def\vv{\varepsilon} \def\rr{\rho}
\def\<{\langle} \def\>{\rangle} \def\GG{\Gamma} \def\gg{\gamma}
  \def\nn{\nabla} \def\pp{\partial} \def\EE{\scr E}
\def\d{\text{\rm{d}}} \def\bb{\beta} \def\aa{\alpha} \def\D{\scr D}
  \def\si{\sigma} \def\ess{\text{\rm{ess}}}
\def\beg{\begin} \def\beq{\begin{equation}}  \def\F{\scr F}
\def\Ric{\text{\rm{Ric}}} \def\Hess{\text{\rm{Hess}}}
\def\e{\text{\rm{e}}} \def\ua{\underline a} \def\OO{\Omega}  \def\oo{\omega}
 \def\tt{\tilde} \def\Ric{\text{\rm{Ric}}}
\def\cut{\text{\rm{cut}}} \def\P{\mathbb P} \def\ifn{I_n(f^{\bigotimes n})}
\def\C{\scr C}      \def\aaa{\mathbf{r}}     \def\r{r}
\def\gap{\text{\rm{gap}}} \def\prr{\pi_{{\bf m},\varrho}}  \def\r{\mathbf r}
\def\Z{\mathbb Z} \def\vrr{\varrho} \def\ll{\lambda}
\def\L{\scr L}\def\Tt{\tt} \def\TT{\tt}\def\II{\mathbb I}
\def\i{{\rm in}}\def\Sect{{\rm Sect}}\def\E{\mathbb E} \def\H{\mathbb H}
\def\M{\scr M}\def\Q{\mathbb Q} \def\texto{\text{o}} \def\LL{\Lambda}
\def\Rank{{\rm Rank}} \def\B{\scr B} \def\i{{\rm i}} \def\HR{\hat{\R}^d}
\def\to{\rightarrow}\def\l{\ell}
\def\8{\infty}\def\X{\mathbb{X}}\def\3{\triangle}
\def\V{\mathbb{V}}\def\M{\mathbb{M}}\def\W{\mathbb{W}}\def\Y{\mathbb{Y}}\def\1{\lesssim}

\def\La{\Lambda}\def\S{\mathbf{S}}
\def\va{\varphi}
\def\l{\lambda}
\def\var{\varphi}
\renewcommand{\bar}{\overline}
\renewcommand{\hat}{\widehat}
\renewcommand{\tilde}{\widetilde}

\maketitle
\begin{abstract}
In this paper, employing the weak convergence method, based on a  variational representation for expected
values of positive functionals of a Brownian motion, we investigate  moderate deviation %(CLT for abbreviation)
 for a class of stochastic differential delay equations with small
noises, where the coefficients are allowed to be highly nonlinear growth with respect to the variables.
Moreover, we obtain the central limit theorem for stochastic differential delay equations which the coefficients are polynomial growth with respect to the delay variables.
\end{abstract}
 \vskip 2cm

\section{Introduction and Main Results}
    There has been extensive literature on the theory of large deviation
principle (LDP) for stochastic differential equations (SDEs) with
small noises  since the pioneer work
due to Freidlin-Wentzell %\cite{freidlin1984random}
\cite{Freidlin}. As we know, the
classical method to show LDP  is based on an approximation argument
and some  exponential-type estimates;
 see, e.g., \cite{CS,CPL,FM,KG,XJ,CW}. %\cite{chen1991probabilities}
% %\cite{Chen}\footnote{provide more references}.
% \cite{CS} study LDP for stochastic reaction diffusion systems with multiplicative noise
% and non-Lipschitz reaction term, \cite{CPL} investigate LDP for some
% parabolic It\^{o} process, and \cite{FM} for Random perturbations of reaction-diffusion equations.
% For more references on this approach, we may refer to \cite{KG,XJ,CW} and references therein.
 As far as the classical method is concerned,
 the exponential-type estimate  is a hard  ingredient to deal with because
 different
 SDEs or stochastic partial differential equations (SPDEs) need different   techniques.
 In recent years, the weak convergence method,
 % which relies on a variational representation
  %for expected values of positive functionals of
 %a Brownian motion or a general poisson random measure
 (see, e.g., %\cite{budhiraja2000variational,budhiraja2011variational}
 \cite{BD,BDV,ADP} and references therein)  has
 been developed  to study LDP problems for diverse setups, where the advantage of this method is that it avoids some exponential
  probability estimates; see, e.g.,
 %\cite{budhiraja2012large,dupuis2011weak}
 \cite{BDP,Dupuis,LW,JD,JR,SSP}
  for SDEs/SPDEs driven by Brownian motion, and
 %\cite{budhiraja2013large,budhiraja2015large}
 \cite{BaoJ,BAC,Bud,ZJ}
 for SDEs/SPDEs driven by jump processes.% and we refer to \cite{LW,JD,JR,SSP} for more details and applications.
  % \cite{Bud,BAC,Dupuis,BDP}
  %and references therein for more details.
  %\cite{budhiraja2008large} investigated the LDP for infinite dimensional stochastic dynamical systems, \cite{bao2015large} proved LDP for neutral functional SDEs with jumps.\\

\smallskip

Recently, numerous mathematicians work on central limit theorem
(CLT); see, e.g., \cite{HK,FP,WZZ}.
%% investigate some central limit theorems for
%super-brownian motion, and the analytic approach also produces new
%results for semilinear SPDEs. \cite{FP} for a class of SPDEs and as
%an application derive this theorem for two widely studied population
%models: super-Brownian motion and the Fleming-Viot process,
% and \cite{LZ} for the moderate deviations and central limit theorem for positive diffusions,
%the proof based on the exponential approximations theorem \cite{DAZ} and Burkholder-Davis-Gundy's inequality.
%Also \cite{ZM} gives more details about CLTs for stochastic dynamic systems.
%On the topic of ,  developments have been made on stochastic partial differential equations , e.g., \cite{fatheddin2015central,li1999some,liming1995moderate} %\cite{FP,WL,ZL}
%and references therein for more details.
 %most of these references are proved the CLT by using the Laplace transform of the process under study.
 %\cite{li2016moderate} %\cite{LZ}
 %investigated the moderate deviations and central limit theorem for positive diffusions, the proof based on the exponential approximations theorem \cite{dembo2009large} %\cite{DAZ}
 %and Burkholder-Davis-Gundy's inequality.
 Since moderate deviation principle (MDP) fills the gap between   CLT scale and LDP scale,
  it has been gained much attention. With regard to MDP, we refer to, e.g., %\cite{budhiraja2016moderate}
  %\cite{Chen} for   independent random vectors in a Banach space,
 \cite{BuD}
   for SDEs driven by a Poisson random measure  in finite and infinite
   dimensions,
   %\cite{li2015moderate}~
   \cite{LiYumeng}
  for stochastic heat equation driven by a Gaussian noise, %\cite{wang2015moderate1}
 \cite{WZZ} for $2$D stochastic Navier--Stokes
  equations,
  %\cite{wang2015moderate}
and  \cite{WZ} for stochastic reaction-diffusion equations with
multiplicative
 noise.
 Specially, \cite{MMM} is devoted to investigate moderate deviations for neutral stochastic differential delay equations with jump, the assumptions in it are those the coefficient is of quadratic growth with respect to the delay variables, inspired this, we try to construct weaker assumptions
 to investigate MDP.  
 %Also see \cite{chen1991probabilities,wang2015moderate1} %\cite{WZZ,Chen}
% and references therein.

\smallskip

It is worthy to point out that most of the literature  focus on MDPs
and CLTs for SDEs with linear growth; see, e.g., \cite{BuD,WZZ}.
Whereas, in the present work, we are interested in MDPs for
a wide range of  SDEs with memory, which allow the coefficients are nonlinear growth with respect to the variables and CLTs  which allow the coefficients to
be of polynomial growth with respect to the delay variables. For
more details on SDEs with memory, we refer to
the monograph %Mohammed
\cite{MS}.% due to Mohammed.
\smallskip

%\section{Framework and main result}
To begin, for any $\varepsilon\in(0,1)$,   %first
consider the following stochastic differential delay equation (SDDE)
 \bg{equation}\label{eq1.1}
\d X^{\ep}(t)=b(X^{\ep}(t),X^{\ep}(t-\tau))\d
t+\ss\ep\sigma(X^{\ep}(t),X^{\ep}(t-\tau))\d W(t),~~~~t>0
\end{equation}
with the initial data $X^\ep(\theta)=\xi(\theta),
\theta\in[-\tau,0]$, where
%$\ep\in(0,1)$,
$b:\mathbb{R}^n\times\mathbb{R}^n\mapsto \mathbb{R}^n$,
$\sigma:\mathbb{R}^n\times\mathbb{R}^n\mapsto \mathbb{R}^{n\times
m}$, and $\{W(t)\}_{t\ge0}$ is an $m$-dimensional Brownian motion
defined on  the filtered probability space  $(\Omega,\mathcal
{F},\{\mathcal {F}_t\}_{t\geq0},\mathbb{P})$.

Intuitively, as $\ep\downarrow 0, \{X^\ep(t)\}_{t\ge 0}$,  the
solution to \eqref{eq1.1}, tends to $\{X^0(t)\}_{t\ge 0}$, which
solves the following deterministic differential delay  equation
\bg{equation}\label{eq1.2} \d X^{0}(t)=b(X^{0}(t),X^{0}(t-\tau))\d
t,\ \ \ \ t>0,\ \ \ X^0(\theta)=\xi(\theta),~~\theta\in[-\tau,0].
\end{equation}

In this paper, we shall investigate deviations of $X^\ep$ from the
deterministic solution $X^0$, as $\ep\downarrow 0$. That is, we are
interested in the asymptotic behavior of the trajectories:
\begin{equation}\label{eq1.3}
Z^\ep(t):=\ff{1}{\ss\ep\lambda(\ep)}(X^\ep(t)-X^0(t)),~~~t\in[0,T],
\end{equation}
in which $\lam(\ep)$ is some deviation scale. In particular,
\begin{enumerate}
\item[(1)] For $\lam(\ep)=\ff{1}{\ss\ep}$, it is corresponding to LDPs;
\item[(2)] For $\lam(\ep)\equiv 1$, it is associated with CLTs;
\item[(3)] For $\lam(\ep)\rightarrow \8$ and $ \ss\ep\lam(\ep)\rightarrow 0$   as $\ep\rightarrow
0,$ it is concerned with MDPs.

\end{enumerate}

%For (1), we refer to, e.g., [Mohammed and Tusheng Zhang] by the
%classical approach \& [Bao-Yuan] by the weak convergence trick.

%In this section, we consider the issue of  case (2).\\

%Next, we introduce some notation.   , $\|\cdot\|_{\H}$ means the
%norm induced by $\langle\cdot,\cdot\rangle$ and $\|\cdot\| $ means
%the usual operator norm.\\
 Let $V:\mathbb{R}^n\times\mathbb{R}^n\mapsto \mathbb{R}_+$ such that
\begin{equation}\label{suo}
V(x,y)\le K(1+|x|^{q}+|y|^{q}),\ \ \ \ x,y\in\mathbb{R}^n
\end{equation}
holds for some constants $K,q\ge 1$. \\
For any
$x_1,x_2,y_1,y_2\in\R^n,$    we assume that
\begin{enumerate}
\item[(\bf H1)]  There exists an $L>0$ such that
$$|b(x_1,y_1)-b(x_2,y_2)|+\|\sigma(x_1,y_1)-\sigma(x_2,y_2)\|_{\rm HS}\le L|x_1-x_2|+V(y_1,y_2)|y_1-y_2|,$$
where $\|\cdot\|_{\rm HS}$ stands  for the Hilbert-Schmidt norm.

\item[(\bf H2)]  $b(\cdot,\cdot)$ is Fr\'{e}chet differentiable w.r.t. each component, and there exists an $L_0>0$ such that
\begin{equation}\label{suo1}
\|\nabla^{(1)}b(x_1,\cdot)-\nabla^{(1)}b(x_2,\cdot)\| \le
L_0|x_1-x_2|,
\end{equation}
and
\begin{equation}\label{suo2}
\|\nabla^{(2)}b(\cdot,y_1)-\nabla^{(2)}b(\cdot,y_2)\| \le
V(y_1,y_2)|y_1-y_2|,
\end{equation}
in which $\nabla^{(i)}b(\cdot,\cdot)$ denotes  the gradient operator
w.r.t. the $i$'th variable.

\end{enumerate}
Under {\bf(H1)}, \eqref{eq1.1} admits a unique strong solution
$\{X^\ep(t)\}_{t\ge-\tau}$ (see, e.g., %\cite[Lemma2.1]{bao2013convergence})
\cite[Lemma 2.1]{Bao}). For $b(x,y)=2x+3y^3$ and $\si(x,y)=4y^2,
x,y\in\R,$ it is easy to see that ({\bf H1}) and ({\bf H2}) hold,
respectively, with $V(x,y)=9(1+x^2+y^2)$ and $L=L_0=2$.

%\begin{rem}\label{2.1}
%{\rm From {(\bf H1)}}, a straightforward calculation gives that for
%any $x,y,z\in\R^n$,
%\begin{equation}\label{eq2.1}
%\begin{split}
%|b(x,y)|+\|\sigma(x,y)\|_{HS}\le C(1+|x|+|y|^{q+1}),
%%&\le|b(x,y)-b(0,0)|+|b(0,0)|\\
%%&\le L|x|+K_1(1+|y|^{q})|y|+|b(0,0)|\\
%%&\le C(1+|x|+|y|+|y|^{q+1}),~~~~~~~~~~x,y\in\mathbb{R}^n\\
%\end{split}
%\end{equation}
% and, from   ({\bf H2}), that
%\begin{equation}\label{66}
%\nabla_x^{(1)}\nabla_x^{(1)}b(\cdot,\cdot)\le
%L_0|x|^2,~~~~\nabla_z^{(2)}\nabla_z^{(2)} b(\cdot,y)\le V(y,y)|z|^2.
%\end{equation}
%%where $Y^\ep=\ff{X^\ep-X^0}{\ss\ep}.$
%\end{rem}

 %Throughout the paper, $c$ is a generic constant, whose value may be different from line to line by convention,
 %     and we use the shorthand notation $a\1 b$ to mean $a\le cb$ for some
 %     $c>0$.

One of our  main results in this paper is presented as below.
\begin{thm}\label{th1.1}
  Under {\rm{(\bf{H1}) and (\bf {H2})}},
\bg{equation*} \E\Big(\sup_{0\le t\le
T}\Big|\ff{X^\ep(t)-X^0(t)}{\ss\ep}-Y(t)\Big|^2\Big)\le c\,\ep
\end{equation*}
for some constant $c>0.$ Herein,  $Y(t)$ solves
\begin{equation*}
\begin{split}
\d Y(t)=\{\nabla_{Y(t)}^{(1)}b(X^{0}(t),X^{0}(t-\tau))+\nabla_{Y(t-\tau)}^{(2)}b(X^{0}(t),X^{0}(t-\tau))\}\d t+\sigma(X^{0}(t),X^{0}(t-\tau))\d W(t)\\
\end{split}
\end{equation*}
where,~$\nabla_{x}^{(i)}$~is the gradient operator along the~$x$~direction,
with the initial value $Y(\theta)\equiv {\bf0_n}$, the zero vector
in $\R^n$, for any $\theta\in[-\tau,0]$, in which
$\{X^0(t)\}_{t\ge-\tau}$ is determined by \eqref{eq1.2}.
\end{thm}
In the sequel, we shall extend Theorems \ref{th1.1}
to SDDEs of neutral type
\begin{equation}\label{eq3.1}
\d\{X^\ep(t)-G(X^\ep(t-\tau))\}=b(X^\ep(t),X^\ep(t-\tau))\d
t+\ss\ep\sigma(X^\ep(t),X^\ep(t-\tau))\d W(t),~~~t>0,
\end{equation}
with the initial data $X^\ep(\theta)=\xi(\theta),
\theta\in[-\tau,0]$, where $G: \R^n\mapsto\R^n$ and the other
parameters are defined exactly  as in \eqref{eq1.1}.

\smallskip

As $\ep\downarrow 0, X^\ep(t)$, the solution to \eqref{eq3.1},
 tends to $X^0(t)$, which solves the deterministic
differential delay equation of neutral type
\begin{equation}\label{suo6}
\d\{X^0(t)-G(X^0(t-\tau))\}=b(X^0(t),X^0(t-\tau))\d t,\ \  t>0,\ \
X^0(\theta)=\xi(\theta),~~\theta\in[-\tau,0].
\end{equation}

\smallskip

Besides ({\bf H1}) and ({\bf H2}),  we further suppose that
\begin{enumerate}
 \item[({\bf H3})] $G(\cdot)$ is Fr\'{e}chet differentiable, and for any $x,y\in\R^n,$
\begin{equation}\label{suo3}
|G(x)-G(y)|\le V(x,y)|x-y|,
\end{equation}
and
\begin{equation}\label{suo4}
|\nabla  G(x)-\nabla G(y)|\le V(x,y)|x-y|,
\end{equation}
where $V(\cdot,\cdot)$ such that \eqref{suo} holds.
\end{enumerate}

Concerning  \eqref{eq3.1},   Theorem \ref{th1.1}   can be
generalized  as below.
\begin{thm}\label{th3.2}
 Under {\rm\textbf{(H1)}}-{\rm\textbf{(H3)}} ,
\bg{equation}\label{love7} \E\Big(\sup_{0\le t\le
T}\Big|\ff{X^\ep(t)-X^0(t)}{\ss\ep}-Y(t)\Big|^2\Big)\le c\ep,
\end{equation}
for some constant $c>0$. Herein, $Y(t)$ solves
\begin{equation}\label{eq3}
\begin{split}
\d\{Y(t)-\nabla_{Y(t-\tau)}G(X^0(t-\tau))\}
&=\{\nabla_{Y(t)}^{(1)}b(X^0(t),X^0(t-\tau))
\\
&~~+\nabla_{Y(t-\tau)}^{(2)}b(X^0(t),X^0(t-\tau))\}\d t\\
&~~+\sigma(X^0(t),X^0(t-\tau))\d W(t),~~~~t>0
\end{split}
\end{equation}
with the initial value $Y(\theta)\equiv {\bf0_n}$, for any
$\theta\in[-\tau,0],$ in which $\{X^0(t)\}_{t\ge -\tau}$ is
determined by \eqref{suo6}.
\end{thm}
The outline of this work is organized as follows: In section 2, we give the proofs of the Theorems \ref{th1.1} and \ref{th3.2}; Section 3 is devoted to the moderate deviation principle for SDDEs, which allow the coefficients are highly nonlinear growth with respect to the variables; In section 4, we give two examples, which the coefficients are polynomial growth with respect to the variables. Throughout the paper,~$C$~is a generic constant, whose value may be different from line to line by convention, and we use the shorthand notation~$a\1b$~to mean~$a\le cb$.~
\section{Proofs of Theorems \ref{th1.1} and \ref{th3.2}}\label{sec2}

%Hereinafter,  we only prove  Theorem \ref{th3.2} since the argument
%of Theorem \ref{th1.1}  is quite similar.

Before we complete  proofs of our main results, we
prepare several auxiliary lemmas. Throughout this section, we point out that $\{X^\ep(t)\},
  \{X^0(t)\}$ and $\{Y(t)\}$ below solve \eqref{eq3.1}, \eqref{suo6}, and
  \eqref{eq3}, respectively.

The lemma below show that $\{X^\ep(t)\}$, $\{X^0(t)\}$,  the solutions to
\eqref{eq3.1}, \eqref{suo6}, respectively, are uniformly bounded in $p$-th moment sense in a
finite horizon. %Although the  proof below is standard in the
%literature, we here give an outline to make the content
%self-contained.

\begin{lem}\label{le3.4} Under  {\rm ({\bf H1})} and \eqref{suo3},
%then for any $p\ge 2$, there exists $C>0$ such that
\begin{equation}\label{love}\E\Big(\sup_{0\le t\le T}|X^\ep(t)|^p\Big)\vee\Big(\sup_{0\le
t\le T}|X^0(t)|^p\Big)\le C,~~~~ p\ge
2,~\ep\in(0,1),\end{equation}
where~$C$~is a constant, which depends on $\|\xi\|_\8:=\sup_{-\tau\le\theta\le0}|\xi(\theta)|$.
\end{lem}
\begin{proof}
 From {\rm{(\bf H1)}}, a straightforward calculation gives that
\begin{equation}\label{eq2.1}
\begin{split}
|b(x,y)|+\|\sigma(x,y)\|_{\rm HS}\11+|x|+|y|^{q+1},~~~x,y\in\R^n.
%&\le|b(x,y)-b(0,0)|+|b(0,0)|\\
%&\le L|x|+K_1(1+|y|^{q})|y|+|b(0,0)|\\
%&\le C(1+|x|+|y|+|y|^{q+1}),~~~~~~~~~~x,y\in\mathbb{R}^n\\
\end{split}
\end{equation}
Hereinafter, $q\ge1$ is
given in \eqref{suo}. Set $r:=1+q$ for notational simplicity and let
$t\in[0,T]$ be arbitrary. From \eqref{eq2.1} and \eqref{suo3},
  we obtain that for any $p\ge 2$ and
  $\ep\in(0,1)$,
\begin{equation*}
\begin{split}
\E\Big(\sup_{0\le s\le t}|X^\ep(s)|^p\Big)&\1\E|\xi(0)-G(\xi(-\tau))|^p+\E\Big(\sup_{0\le s\le t}|G(X^\ep(s-\tau)|^p\Big)\\
&~~+\E\Big(\sup_{0\le s\le t}\Big|\int_0^s b(X^\ep(u),X^\ep(u-\tau))\d u\Big|^p\Big)\\
&~~+\ep^{p/2}\E\Big(\sup_{0\le s\le t}\Big|\int_0^s\sigma(X^\ep(u),X^\ep(u-\tau))\d W(u)\Big|^p\Big)\\
&\11+\|\xi\|_\8^{pr}+\E\Big(\sup_{0\le s\le  (t-\tau)\vee0 }|X^\ep(s)|^{pr}\Big)\\
&~~+\int_0^t\E|b(X^\ep(s),X^\ep(s-\tau))|^p\d s+\ep^{p/2}\int_0^t\E\|\sigma(X^\ep(s),X^\ep(s-\tau))\|_{\rm HS}^p\d s\\
&\11+\|\xi\|_\8^{pr}+\int_0^t\E|X^\ep(s)|^p\d
s+\int_0^{(t-\tau)\vee0}\E|X^\ep(s)|^{pr}\d s+\E\Big(\sup_{0\le s\le
(t-\tau)\vee0}|X^\ep(s)|^{pr}\Big),
\end{split}
\end{equation*}
where $a\vee b:=\max\{a,b\}$ for $a,b\in\R.$ By the Gronwall
inequality, one has
\begin{equation}\label{e1}
\E\Big(\sup_{0\le s\le t}|X^\ep(s)|^p\Big)\1
1+\|\xi\|_\8^{pr}+\E\Big(\sup_{0\le s\le
(t-\tau)\vee0}|X^\ep(s)|^{pr}\Big)+\int_0^{(t-\tau)\vee0}\E|X^\ep(s)|^{pr}\d
s.
\end{equation}
 Let
$$p_i=([T/\tau]+2-i)p r^{[T/\tau]+1-i},~~~~~i=1,2,\cdots,[T/\tau]+1.$$
It is easy to see that $p_i\ge 2$ such that
$$rp_{i+1}<p_i~~~~~\mbox{and}~~~~~p_{[T/\tau]+1}=p,~~~i=1,2,\cdots,[T/\tau].$$
We obtain from \eqref{e1} that
\begin{equation*}
\E\Big(\sup_{0\le s\le\tau}|X(s)|^{p_1}\Big)\11+\|\xi\|_\8^{p_1r}.
\end{equation*}
This further yields by H\"older's inequality that
\begin{equation*}
\begin{split}
\E\Big(\sup_{0\le s\le 2\tau}|X^\ep(s)|^{p_2}\Big)
%&\le C\Big(1+\E(|X^\ep(s-\tau)|^{p_2r})+\E\int_0^{2\tau}(|X^\ep(s-\tau)|^{p_2r})\d s\Big)\\
&\11+\E\Big(\sup_{0\le s\le \tau} |X^\ep(s)|^{p_2r}\Big)+\int_0^{\tau}\E|X^\ep(s)|^{p_2r}\d s\\
&\11+\Big(\E\Big(\sup_{0\le s\le \tau}|X^\ep(s)|^{p_1}\Big)\Big)^{\ff{p_2r}{p_1}}+\int_0^{\tau}\Big(\E|X^\ep(s)|^{p_1}\Big)^{\ff{p_2r}{p_1}}\d s\\
&\11+\|\xi\|_\8^{p_2r^2}.
\end{split}
\end{equation*}
Repeating the previous procedures yields that
\begin{equation}\label{love2}\E\Big(\sup_{0\le t\le T}|X^\ep(t)|^p\Big)\1
1+\|\xi\|_\8^{p(1+q)^{[T/\tau]+1}},~~~~ p\ge 2,~\ep\in(0,1),\end{equation} which
further leads to
$$\Big(\sup_{0\le t\le T}|X^0(t)|^p\Big)\1 1+\|\xi\|_\8^{p(1+q)^{[T/\tau]+1}},~~~~ p\ge 2$$
by letting $\ep$ go to zero. The proof is therefore complete.
\end{proof}

The following lemma provides the order of deviation between $X^\ep$
and $X^0$.

\begin{lem}\label{le3.2}
 Under  {\rm ({\bf H1})} and \eqref{suo3}, for any $p\ge2$ there is a constant $C_{p,T}>0$, such that
\bg{equation}\label{eq1}
 \E\Big(\sup_{0\le t\le
T}\Big|\ff{X^\ep(t)-X^0(t)}{\ss\ep}\Big|^p\Big)\le C_{p,T},
\end{equation}
\end{lem}
\begin{proof}
For notational simplicity, set \begin{equation}\label{love5} Y^\ep
:= \ff{X^\ep-X^0}{\ss\ep}.
\end{equation} Since \eqref{eq3.1} and
\eqref{suo6} share the same initial value, one has
$Y^\ep(\theta)\equiv{\bf 0_n}$ for any $\theta\in[-\tau,0].$ In
terms of \eqref{love}, one gets from \eqref{suo} that, for each
$l\ge1$, there exits a constant $C_{l,T}>0$ such that
\begin{equation}\label{love1}
\E\Big(\sup_{-\tau\le t\le T} V^l(X^\ep(t),X^0(t))\Big)\le C_{l,T}.
\end{equation}
By %the
%Due to $X^\ep(\theta)=X^0(\theta)=\xi,~\xi\in[-\tau,0],$
 the elementary inequality:
 \begin{equation}\label{eq2.3}
(a_1+a_2+\cdots+a_m)^l\le
m^{l-1}(a_1^l+a_2^l+\cdots+a_m^l),~~~l\ge1,~a_i\ge 0,
\end{equation}
the B-D-G inequality as well as the H\"older inequality, we obtain
{\rm ({\bf H1})} and \eqref{suo3} that for $p\ge 2$ and $t\in[0,T]$
\begin{equation*}
\begin{split}
\E\Big(\sup_{0\le s\le t}|Y^\ep(s)|^p\Big)&\1\E\Big(\sup_{0\le s\le t}\Big|\ff{G(X^\ep(s-\tau))-G(X^0(s-\tau))}{\ss\ep}\Big|^p\Big)\\
&~~+\E\Big(\sup_{0\le s\le t}\Big|\int_0^s\ff{b(X^{\ep}(r),X^{\ep}(r-\tau))-b(X^0(r),X^0(r-\tau))}{\ss\ep}\d r\Big|^p\Big)\\
&~~+\E\Big(\sup_{0\le s\le t}|\int_0^s\sigma(X^{\ep}(r),X^{\ep}(r-\tau))\d W(r)|^p\Big)\\
&\1\Big(\E\Big(\sup_{-\tau\le s\le t-\tau}V^{2p}(X^\ep(s),X^0(s))\Big)\Big)^{\ff{1}{2}}\Big(\E\Big(\sup_{0\le s\le (t-\tau)\vee0}|Y^\ep(s)|^{2p}\Big)\Big)^{\ff{1}{2}}\\
&~~+\int_0^t\E|Y^\ep(s)|^p\d s+\int_0^{(t-\tau)\vee0}(\E
V^{2p}(X^\ep(s),X^0(s)))^{\ff{1}{2}}
(\E|Y^\ep(s)|^{2p})^{\ff{1}{2}}\d s\\
&~~+\int_0^t\{1+\E|X^\ep(s)|^p+\E|X^\ep(s-\tau)|^{p(q+1)}\}\d s\\
%&\1\ep^{-\ff{p}{2}}\int_0^t\E\Big(|X^\ep(s)-X^0(s)|+V(X^\ep(s-\tau),X^0(s-\tau))
%|X^\ep(s-\tau)-X^0(s-\tau)|\Big)^p\d s\\
%&~~+\int_0^t\E\Big(1+|X^\ep(s)|+|X^\ep(s-\tau)|^{q+1}\Big)^p\d s\\
%&\1\Big(1+\ep^{-\ff{p}{2}}\int_0^t\E|X^\ep(s)-X^0(s)|^p\d s\\
%&~~+\ep^{-\ff{p}{2}}\int_0^t(\E V^{2p}(X^\ep(s-\tau),X^0(s-\tau)))^{\ff{1}{2}}
%(\E|X^\ep(s-\tau)-X^0(s-\tau)|^{2p})^{\ff{1}{2}}\d s\Big)\\
&\11+\Big(\E\Big(\sup_{0\le s\le
(t-\tau)\vee0}|Y^\ep(s)|^{2p}\Big)\Big)^{\ff{1}{2}}
+\int_0^t\E|Y^\ep(s)|^p\d
s+\int_0^{(t-\tau)\vee0}(\E|Y^\ep(s)|^{2p})^{\ff{1}{2}}\d s,
\end{split}
\end{equation*}
where in the last display we have used \eqref{love} and
\eqref{love1}. So,   Gronwall's inequality gives that
\begin{equation}\label{eq2.4}
\E\Big(\sup_{0\le s\le t}|Y^\ep(s)|^p\Big)\11+\Big(\E\Big(\sup_{0\le
s\le
(t-\tau)\vee0}|Y^\ep(s)|^{2p}\Big)\Big)^{\ff{1}{2}}+\int_0^{(t-\tau)\vee0}(\E|Y^\ep(s)|^{2p})^{\ff{1}{2}}\d
s.
\end{equation}
In what follows, by mimicking the argument of \eqref{love2}, the
proof of Lemma \ref{le3.2} can be done.
%Let
%$$p_i=([T/\tau]+2-i)p 2^{[T/\tau]+1-i},~~~~~i=1,2,\cdots,[T/\tau]+1.$$
%It is easy to see that $p_i\ge 2$ such that
%$$2p_{i+1}<p_i~~\mbox{and}~~p_{[T/\tau]+1}=p,~i=1,2,\cdots,[T/\tau].$$
%From \eqref{eq2.4}, one has
%\begin{equation*}
%\E\Big(\sup_{0\le s\le\tau}|Y^\ep(s)|^{p_1}\Big)\le c.
%\end{equation*}
%This, by combining \eqref{eq2.4} with the H\"{o}lder inequality,
%  gives that
%\begin{equation*}
%\begin{split}
%\E\Big(\sup_{0\le s\le 2\tau}|Y^\ep(s)|^{p_2}\Big)&\11+\int_0^{2\tau}(\E|Y^\ep(s-\tau)|^{2p_2})^{\ff{1}{2}}\d s\\
%&\11+\int_0^{\tau}(\E|Y^\ep(s)|^{p_1})^{\ff{p_2}{p_1}}\d s\\
%&\le c_0,
%\end{split}
%\end{equation*}
%for some $c_0>0.$ The desired assertion follows by repeating the
%previous procedures.
\end{proof}
\begin{lem}\label{le3.3}
 Under   {\rm ({\bf H1})} and  {\rm ({\bf H2})}, for any $p\ge2$ there exists a constant $\tt C_{p,T}>0$ such that
\begin{equation}\label{eq2}
\E\Big(\sup_{0\le t\le T}|Y(t)|^p\Big)\le \tt C_{p,T}.
\end{equation}
\end{lem}
\begin{proof}
In view of {\bf (H1)}, we find that for any $x,y,z\in\R^n,$
\begin{equation}\label{love3}
|\nabla_z^{(1)}b(x,y)|\le L|z|~~~~~~~\mbox{ and
}~~~~~~~~|\nabla_z^{(2)}b(x,y)|\le V(y,y)|z|.
\end{equation}
On the other hand, from \eqref{suo3} we have
\begin{equation}\label{love4}
|\nabla_z G(y)|\le V(y,y)|z|,~~~~~~y,z\in\R^n.
\end{equation}
Recall that $Y(\theta)={\bf0_n}$ for any $ \theta\in[-\tau,0]$.
Then, by B-D-G's inequality and H\"older's inequality, in addition
to \eqref{eq2.1}, \eqref{love3} as well as \eqref{love4}, we deduce
that for any $p\ge2$
\begin{align*}
\E\Big(\sup_{0\le s\le t}|Y(s)|^p\Big)
&\1\E\Big(\sup_{0\le s\le
(t-\tau)\vee0}V^{p}(X^0(s),X^0(s))|Y(s)|^{p}\Big)\\
&~~+\int_0^t\E\{1+|X^0(s)|+|X^0(s-\tau)|^{q+1}\}^p\d s\\
&~~+\int_0^t\E|Y(s)|^p\d s+\int_0^{(t-\tau)\vee0}\E V^{p}(X^0(s),X^0(s))|Y(s)|^{p}\d s\\
&\11+\E\Big(\sup_{0\le s\le
t-\tau}|Y(s)|^{p}\Big)+\int_0^t\E|Y(s)|^p\d
s+\int_0^{(t-\tau)\vee0}\E|Y(s)|^{p}\d s\\
&\11+\E\Big(\sup_{0\le s\le
t-\tau}|Y(s)|^{p}\Big)+\int_0^t\E|Y(s)|^p\d s,
\end{align*}
where in the last second inequality we have utilized \eqref{love}.  Then,
the desired assertion is available by the Gronwall's inequality and an induction argument.
\end{proof}

With Lemmas \ref{le3.4}-\ref{le3.3} in hand, we are now in position to complete the proof of Theorem \ref{th3.2}.
%\noindent{\bf Proof of Theorem \ref{th3.2}}\\
\begin{proof}[{\bf Proof of Theorem \ref{th3.2}}]
Let $Y^\ep$ be defined as in \eqref{love5}. Thanks to
$Y^\ep(\theta)=Y(\theta)\equiv{\bf0_n}$ for any
$\theta\in[-\tau,0]$, it follows that
\begin{equation*}
\begin{split}
Y^\ep(t)-Y(t)&=\ff{G(X^\ep(t-\tau))-G(X^0(t-\tau))}{\ss\ep}-\nabla_{Y^\ep(t-\tau)}G(X^0(t-\tau))\\
&~~+\int_0^t\Big\{\ff{b(X^{\ep}(s),X^{0}(s-\tau))-b(X^{0}(s),X^{0}(s-\tau))}{\ss\ep}
-\nabla_{Y^\ep(s)}^{(1)}b(X^{0}(s),X^{0}(s-\tau))\Big\}\d s\\
&~~+\int_0^t\Big\{\ff{b(X^{\ep}(s),X^{\ep}(s-\tau))-b(X^{\ep}(s),X^{0}(s-\tau))}{\ss\ep}
-\nabla_{Y^\ep(s-\tau)}^{(2)}b(X^{\ep}(s),X^{0}(s-\tau))\Big\}\d s\\
&~~+\int_0^t\{\sigma(X^{\ep}(s),X^{\ep}(s-\tau))-\sigma(X^{0}(s),X^{0}(s-\tau))\}\d W(s)\\
&~~+\nabla_{Y^\ep(t-\tau)-Y(t-\tau)}G(X^0(t-\tau))\\
&~~+\int_0^t\nabla_{Y^\ep(s)-Y(s)}^{(1)}b(X^{0}(s),X^{0}(s-\tau))\d s\\
&~~+\int_0^t\nabla_{Y^\ep(s-\tau)-Y(s-\tau)}^{(2)}b(X^{\ep}(s),X^{0}(s-\tau))\d s\\
&~~+\int_0^t\Big\{\nabla_{Y(s-\tau)}^{(2)}b(X^{\ep}(s),X^{0}(s-\tau))
-\nabla_{Y(s-\tau)}^{(2)}b(X^{0}(s),X^{0}(s-\tau))\Big\}\d s\\
&=:\sum_{i=1}^8I_i(t).
\end{split}
\end{equation*}
%Noticing that $Y^\ep(s)=Y(s)\equiv 0,s\in[-\tau,0]$,
%We shall consider $J_i(t):=\E\Big(\sup_{0\le s\le t}|I_i(s)|^2\Big),i=1,2,\cdots,6,$ respectively.\\
Observe from {\bf(H2)} that for any $x,y,z\in\R^n$,
\begin{equation}\label{66}
\begin{split}
&|\nabla_z^{(1)}\nabla_z^{(i)}b(x,y)|\le L_0|z|^2,~~~i=1,2,
~~~|\nabla_z^{(2)}\nabla_z^{(2)} b(x,y)|\le V(y,y)|z|^2,
\end{split}
\end{equation}
and from {\bf(H3)} that
\begin{equation}\label{love66}
|\nabla_y\nabla_yG(x)|\le V(x,x)|y|^2.
\end{equation}
For notational simplicity, set $J_i(t):=\E\Big(\sup_{0\le s\le
t}|I_i(s)|^2\Big)$. To achieve the desired assertion, in
what follows we intend to estimate $J_i(t)$ one-by-one. By  Taylor's
expansion and H\"older's inequality,   together with \eqref{love4}
and \eqref{love66}, we infer from \eqref{love} and \eqref{eq2} that
\begin{equation*}
\begin{split}
J_1(t)+J_5(t)&\1\ep\E\Big(\sup_{0\le s\le (t-\tau)\vee0}|\nabla_{Y^\ep(s)}\nabla_{Y^\ep(s)}G(X^0(s)+u^\ep(s))|^2\Big)\\
&~~+\E\Big(\sup_{0\le s\le (t-\tau)\vee0}|\nabla_{Y^\ep(s)-Y(s)}G(X^0(s))|^2\Big)\\
&\1\ep\E\Big(\sup_{0\le s\le (t-\tau)\vee0}V^2(X^0(s)+u^\ep(s),X^0(s)+u^\ep(s))|Y^\ep(s)|^{4}\Big)\\
&~~+\E\Big(\sup_{0\le s\le (t-\tau)\vee0}|Y^\ep(s)-Y(s)|^2V^2(X^0(s),X^0(s))\Big)\\
&\1\ep\Big(\E\Big(\sup_{0\le s\le (t-\tau)\vee0}V^{4}(X^0(s)+u^\ep(s),X^0(s)+u^\ep(s))\Big)\Big)^{1/2}\Big(\E\Big(\sup_{0\le s\le (t-\tau)\vee0}|Y^\ep(s)|^{8}\Big)\Big)^{\ff{1}{2}}\\
&~~+\E\Big(\sup_{0\le s\le (t-\tau)\vee0}|Y^\ep(s)-Y(s)|^{2}\Big)\\
&\1\ep+\E\Big(\sup_{0\le s\le
(t-\tau)\vee0}|Y^\ep(s)-Y(s)|^{2}\Big),
\end{split}
\end{equation*}
where
\begin{equation}\label{qiang}
u^\ep(s):=\theta^\ep(s)(X^{\ep}(s)-X^{0}(s)),~~~~s\in[0,t]
\end{equation}
with
$\theta^\ep\in(0,1)$ being a random variable. Also,
with the aid of the Taylor expansion and the H\"older inequality, along with \eqref{love}, \eqref{eq1}, \eqref{love1}, \eqref{eq2}, \eqref{66}, and \eqref{qiang}, we derive that  %there exists a random variable
%$\theta^\ep(s)\in(0,1)$,
%we obtain from  \textbf{(H1)-(H2)} and
%\eqref{66} that
\begin{equation*}
\begin{split}
  J_2(t)+J_3(t)+J_8(t)&\1 \ep\int_0^t\E|\nabla_{Y^\ep(s)}^{(1)}
\nabla_{Y^\ep(s)}^{(1)}b(X^{0}(s)+u^\ep(s),X^{0}(s-\tau))|^2\d s\\
&~~~+\ep\int_0^t\E|\nabla_{Y^\ep(s-\tau)}^{(2)}
\nabla_{Y^\ep(s-\tau)}^{(2)}b(X^{\ep}(s),X^{0}(s-\tau)+u^\ep(s-\tau))|^2\d s\\
&~~+\ep\E\int_0^t|\nabla_{Y^\ep(s)}^{(1)}\nabla_{Y(s-\tau)}^{(2)}b(X^{0}(s)+u^\ep(s),X^{0}(s-\tau))|^2\d s\\
&\1\ep\int_0^t\{\E|Y^\ep(s)|^{4}+\E|Y^\ep(s)\cdot Y(s-\tau)|^{2}\}\d s\\
&~~+\ep\int_0^{(t-\tau)\vee0}(\E (V^{4}(X^{0}(s)+u^\ep(s),X^0(s)+u^\ep(s))))^{1/2}(\E|Y^\ep(s)|^{8})^{1/2}\d s\\
&\1\ep.
\end{split}
\end{equation*}
%\begin{equation*}
%\begin{split}
%  J_2(t)+J_3(t)&\1 \ep\int_0^t\E|\nabla_{Y^\ep(s)}^{(1)}
%\nabla_{Y^\ep(s)}^{(1)}b(X^{0}(s)+u^\ep(s),X^{0}(s-\tau))|^2\d s\\
%&~~~+\ep\int_0^t\E|\nabla_{Y^\ep(s-\tau)}^{(2)}
%\nabla_{Y^\ep(s-\tau)}^{(2)}b(X^{\ep}(s),X^{0}(s-\tau)+u^\ep(s-\tau))|^2\d s\\
%&\1\ep\int_0^t\E|Y^\ep(s)|^4\d s+\ep\int_0^{t-\tau}\E V^2(X^{0}(s)+u^\ep(s),X^0(s)+u^\ep(s))|Y^\ep(s)|^4\d s\\
%&\1\ep.
%\end{split}
%\end{equation*}
%\begin{equation*}
%\begin{split}
%J_2(t)&\1 \ep\int_0^t\E|\nabla_{Y^\ep(s-\tau)}^{(2)}
%\nabla_{Y^\ep(s-\tau)}^{(2)}b(X^{\ep}(s),X^{0}(s-\tau)+u^\ep(s-\tau))|^2\d s\\
%&\1\ep\int_0^t\E V^2(X^{0}(s-\tau)+u^\ep(s-\tau),X^0(s-\tau)+u^\ep(s-\tau))|Y^\ep(s-\tau)|^4\d s\\
%&\1\ep\int_0^t\Big(\E V^4(X^{0}(s-\tau)+u^\ep(s-\tau),X^0(s-\tau)+u^\ep(s-\tau))\Big)^{\ff{1}{2}}
%\Big(\E|Y^\ep(s-\tau)|^8\Big)^{\ff{1}{2}}\d s\\
%&\1\ep,\\
%\end{split}
%\end{equation*}
Using B-D-G's inequality and H\"older's inequality and taking
\eqref{eq1} and \eqref{love1}  into consideration yields that
\begin{equation*}
\begin{split}
J_4(t)&\1 \int_0^t\E\|\sigma(X^{\ep}(s),X^{\ep}(s-\tau))-\sigma(X^{0}(s),X^{0}(s-\tau))\|_{\rm HS}^2\d s\\
&\1\int_0^t\E|X^{\ep}(s)-X^{0}(s)|^2\d s
+\int_0^{(t-\tau)\vee0}\E V^2(X^{\ep}(s),X^0(s))|X^{\ep}(s)-X^{0}(s)|^2\d s\\
&\1\ep \int_0^t\E|Y^{\ep}(s)|^2\d s+\ep
\int_0^{(t-\tau)\vee0}(\E V^{4}(X^{\ep}(s),X^0(s)))^{\ff{1}{2}}(\E|Y^{\ep}(s)|^{4})^{\ff{1}{2}}\d s \\
&\1\ep.
\end{split}
\end{equation*}
With the help of \eqref{love} and \eqref{love3},
\begin{equation*}
\begin{split}
\sum_{i=6}^7J_i(t)
&\1\int_0^t\E|Y^\ep(s)-Y(s)|^2\d s+\int_0^{(t-\tau)\vee0}\E (V^2(X^0(s),X^0(s))|Y^\ep(s)-Y(s)|^2)\d s\\
&\1\int_0^t\E|Y^\ep(s)-Y(s)|^2\d s.
\end{split}
\end{equation*}
%\begin{equation*}
%\begin{split}
%J_5(t)&\1\E\int_0^t|\nabla_{Y^\ep(s-\tau)-Y(s-\tau)}^{(2)}b(X^{0}(s),X^{0}(s-\tau))|^2\d s\\
%&\1\int_0^t\E V^2(X^0(s-\tau),X^0(s-\tau))|Y^\ep(s-\tau)-Y(s-\tau)|^2\d s\\
%&\1\int_0^t(\E V^4(X^0(s-\tau),X^0(s-\tau)))^{\ff{1}{2}}(\E|Y^\ep(s-\tau)-Y(s-\tau)|^4)^{\ff{1}{2}}\d s\\
%&\1\int_0^t(\E|Y^\ep(s-\tau)-Y(s-\tau)|^4)^{\ff{1}{2}}\d s.\\
%\end{split}
%\end{equation*}
%Also,  by virtue of the Taylor expansion theorem and  ({\bf H2}),
%\begin{equation*}
%\begin{split}
%J_8(t)&\1\E \int_0^t|\nabla_{Y(s-\tau)}^{(2)}b(X^{\ep}(s),X^{0}(s-\tau))
%-\nabla_{Y(s-\tau)}^{(2)}b(X^{0}(s),X^{0}(s-\tau))|^2\d s\\
%&\1\ep\E\int_0^t|\nabla_{Y^\ep(s)}^{(1)}\nabla_{Y(s-\tau)}^{(2)}b(X^{0}(s),X^{0}(s-\tau))|^2\d s\\
%&\1\ep\int_0^t\E|Y^\ep(s)|^2V^2(X^0(s-\tau),X^0(s-\tau))|Y(s-\tau)|^2\d s\\
%&\1\ep.
%\end{split}
%\end{equation*}
Hence, we arrive at
\begin{equation}\label{eq3.2}
\begin{split}
&\E\Big(\sup_{0\le s\le
t}|Y^\ep(s)-Y(s)|^2\Big)\1\ep+\E\Big(\sup_{0\le s\le
(t-\tau)\vee0}|Y^\ep(s)-Y(s)|^{2}\Big).
\end{split}
\end{equation}
Then, with the induction argument, we get the desired assertion.
 \end{proof}

\begin{proof}[{\bf Proof of Theorem \ref{th1.1}}]
The proof of Theorem \ref{th1.1} can be done by taking
$G\equiv{\bf0_n}$ in the argument of Theorem \ref{th3.2}.
\end{proof}
\section{Moderate deviation principle}\label{sec3}
In what follows, we recall   some basic notions concerned with LDPs
(see, e.g., \cite[Chapter 1]{DAZ}).
 Let $\S$ be a Polish space (i.e., a complete
separable metrizable topological
 space)  and $\mathscr{B}(\S)$ the Borel $\sigma$-algebra generated by all open sets in
 $\S$.
\begin{defn}\label{de4.1}
  A function $I:\S\mapsto[0,\8]$ is called a rate
function, if for each $M<\8$, the level set $\{x\in\S:I(x)\le M\}$
is a
  compact subset of $\S$. %A family of positive numbers $\{\lam(\ep)\}_{\ep>0}$ is called a speed function if $\lam(\ep)\rightarrow +\8$  as $\ep\rightarrow 0$.
 \end{defn}
 \begin{defn}\label{de4.2}
 A family $\{X^\ep\}$ of~$\S$-valued random variables defined on the
 probability space $(\Omega,\mathcal {F},\mathbb{P})$ is said to satisfy the large deviation
principle
  on $\S$ with the rate function $I$ and the speed function $\{\lam(\ep)\}_{\ep>0}$, if the following conditions
  hold:
\begin{enumerate}
\item[(i)] {\rm(Upper bound)} For each closed subset $F$ of $\S$,
$$\limsup_{\ep\rightarrow 0} \ff{1}{\lam(\ep)}\log\P(X^\ep\in F)\le -\inf_{x\in F}I(x);$$
\item[(ii)] {\rm(Lower bound)} For each open subset $G$ of $\S$,
$$\liminf_{\ep\rightarrow 0} \ff{1}{\lam(\ep)}\log\P(X^\ep\in G)\ge -\inf_{x\in G}I(x).$$
\end{enumerate}
\end{defn}

%Fix $T>0$ and set $H:=L^2([0,T];\R^m)$.
 % For any $N>0$, define
%$\H^N=\{h\in\H;\|h\|_{\H}\le N\},$ where
%$%\|h\|_{\H}:=(\int_0^T|h(t)|^2\d t)^{1/2}$.
we need to introduce some notation.\\

Define the Cameron-Martin space~$\bar{\mathscr{A}}$~by
\begin{equation}\label{eq1.10}
\begin{split}
\bar{\mathscr{A}}=&\big\{h:[0,T]\mapsto\mathbb{R}^m|h~\mbox{is}~
\bar{\mathscr{P}}/\mathscr{B}(\mathbb{R}^m)-\mbox{measurable},\\
&h(t)=\int_0^t\dot{h}(s)\d s,t\in[0,T],~\mbox{and}~\int_0^T|\dot{h}(s)|^2\d s<\8,~~\bar{\mathbb{P}}-a.s.\big\},
\end{split}
\end{equation}
where the dot denotes the generalized derivative, and define
\begin{equation}\label{eq1.11}
L_T(h):=\frac{1}{2}\int_0^T|\dot{h}(s)|^2\d s.
\end{equation}
For each~$N>0$,~let
$$S_N=\{h:[0,T]\mapsto\mathbb{R}^m:L_T(h)\le N\}.$$
Let~$S=\cup_{N\ge1}S_N$~and~$\bar{\mathscr{A}}_N=\{h\in\bar{\mathscr{A}}:h(\omega)\in S_N,~\bar{\mathbb{P}}-a.s.\}$.~\\

Recall the SDDEs of neutral type
\begin{equation}\label{eq33.1}
\d\{X^\ep(t)-G(X^\ep(t-\tau))\}=b(X^\ep(t),X^\ep(t-\tau))\d
t+\ss\ep\sigma(X^\ep(t),X^\ep(t-\tau))\d W(t),~~~t>0,
\end{equation}
In this section, we consider the Moderate deviation principles for this kind of SDDEs, which allow the coefficients are nonlinear growth with all the variables, specifically, we assume the following assumptions on the coefficients of \eqref{eq33.1} hold.~For more details, please refer to \cite[Theorem1.1]{BBB} and
\cite[Corollary 3.5.]{MM}.\\
Recall the polynomial function~$V(x,y)\le K(1+|x|^{q_1}+|y|^{q_1})$.
\begin{enumerate}
%\item[(\bf A1)] (The Local Lipschitz Condition)For each integer~$k\ge1$~
%There exists a positive constant~$L_k>0$~such that
%\begin{equation}\label{eq1.4}
%\begin{split}
%&|b(x_1,y_1)-b(x_2-y_2)|^2+\|\sigma(x_1,y_1)-\sigma(x_2,y_2)\|_{HS}^2\\
%&\le L_k(|x_1-x_2|^2+|y_1-y_2|^2).
%\end{split}
%\end{equation}
%for all~$|x_1|\vee|x_2|\vee|y_1|\vee|y_2|\le k$,~
\item[(\bf A1)] Assume there are constants~$q>p\ge2, a_1\ge0, a_2>a_3\ge 0$~
and~$a_4>a_5>0$~such that
\begin{equation}\label{eq1.5}
\begin{split}
&p|x-G(y)|^{p-2}\Big(\langle x-G(y),b(x,y)\rangle
+\frac{(p-1)}{2}\|\sigma(x,y)\|_{HS}^2\Big)\\
&\le a_1-a_2|x|^p+a_3|y|^p-a_4|x|^q+a_5|y|^q
\end{split}
\end{equation}
for all~$(x,y)\in\mathbb{R}^n\times\mathbb{R}^n$.~
\item[(\bf A2)]~$b$~and~$\sigma$~are continuous and bounded on bounded subsets
 of~$\mathbb{R}^n\times\mathbb{R}^n$,~
\begin{equation}\label{eqq2}
|G(x)-G(y)|\le V(x,y)|x-y|.
\end{equation}
We also assume the assumptions for the gradients of the coefficients.
\item[(\bf A3)] $b(\cdot,\cdot)$ is Fr\'{e}chet differentiable w.r.t. each component, and the gradient satisfy follows,
\begin{equation}\label{eq1.6}
\|\nabla^{(1)}b(x_1,\cdot)-\nabla^{(1)}b(x_2,\cdot)\| \le
V(x_1,x_2)|x_1-x_2|,~~x_1,x_2\in\mathbb{R}^n,
\end{equation}
and
\begin{equation}\label{eq1.7}
\|\nabla^{(2)}b(\cdot,y_1)-\nabla^{(2)}b(\cdot,y_2)\| \le
V(y_1,y_2)|y_1-y_2|,~~y_1,y_2\in\mathbb{R}^n.
\end{equation}
and
\begin{equation}\label{eq11.7}
\|\nabla G(x)-\nabla G(y)\| \le
V(x,y)|x-y|,~~x,y\in\mathbb{R}^n.
\end{equation}
\end{enumerate}
\begin{rem}
 Reference \cite{MMM} investigated the moderate deviation principle for this kind of neutral stochastic
  differential delay equations with jumps. Inspirited by the assumptions therein, we give more weakly assumptions {\rm\textbf{(A1)-(A2)}}, specifically,
the drift and diffusion coefficients in this paper
 are allowed to be highly nonlinear
 growth with respect to the variables. We remark also that
 under {\rm\textbf{(A3)}}, the gradients are also polynomial growth with respect to variables.
 The assumptions in \cite{MMM} are the special case of this work.
\end{rem}

The main result of this section is stated as follows.
%\begin{thm}\label{th1.2} Under {\rm\textbf{(H1)}} and {\rm\textbf{(H2)}},   $\{Z^\ep\}$,   defined by \eqref{eq1.3},
% satisfies an LDP in $C([0,T];\R^n)$  with speed $\lam^2(\ep)$ such that $\lam(\ep)\rightarrow \8$ and $ \ss\ep\lam(\ep)\rightarrow 0$   as $\ep\rightarrow
%0 $ and the rate function given by
%\begin{equation}\label{eq4.1}
%I(f)=\inf_{h\in \mathcal{S}_f}L_T(h).
%\end{equation}
%Herein, $\mathcal{S}_f:=\Big\{h\in S:f=\mathscr{G}^0\big(\int_0^\cdot\dot{h}(s)\d s\big)\Big\},$
% and $Z^h(t)$
%solves the deterministic differential delay  equation
%\begin{equation}\label{eq4.2}
%\begin{split}
%\d Z^h(t)&=\{\sigma(X^0(t),X^0(t-\tau))h(t)+\nabla_{Z^h(t)}^{(1)}b(X^0(t),X^0(t-\tau))\\
%&~~~+\nabla_{Z^h(t-\tau)}^{(2)}b(X^0(t),X^0(t-\tau))\}\d t
%\end{split}
%\end{equation}
%with the initial value $Z^h(\theta)\equiv {\bf0_n}$, for any
%$\theta\in[-\tau,0]$.
%\end{thm}
%On the other hand, Theorem \ref{th1.2} can   be extended as follows.

\begin{thm}\label{th5.1} Under {\rm\textbf{(A1)-(A3)}}, $\{Z^\ep\}$, defined by \eqref{eq1.3}, satisfies an LDP in $C([0,T];\R^n)$  with
speed $\lam^2(\ep)$  such that $\lam(\ep)\rightarrow \8$ and $
\ss\ep\lam(\ep)\rightarrow 0$   as $\ep\rightarrow 0 $ and the rate
function given by
\begin{equation}\label{eq5.4}
I(f)=\inf_{h\in \mathcal{S}_f}L_T(h),
\end{equation}
Herein, $\mathcal{S}_f:=\Big\{h\in S:f=\mathscr{G}^0\big(\int_0^\cdot\dot{h}(s)\d s\big)\Big\},$ and $Z^h(t)$ solves the  deterministic  differential delay
 equation of neutral type
 \begin{equation}\label{eq5.2}
 \begin{split}
 \d\{Z^h(t)-\nabla_{Z^h(t-\tau)}G(X^0(t-\tau))\}
 &=\{\sigma(X^0(t),X^0(t-\tau))\dot{h}(t)+\nabla_{Z^h(t)}^{(1)}b(X^0(t),X^0(t-\tau))\\
 &~~~+\nabla_{Z^h(t-\tau)}^{(2)}b(X^0(t),X^0(t-\tau))\}\d t
 \end{split}
 \end{equation}
with the initial value $Z^h(\theta)\equiv {\bf0_n}$, for any
$\theta\in[-\tau,0]$.
\end{thm}

For the solution $X^\ep(t)$ to \eqref{eq33.1}, by the Yamada-Watanabe theorem % \cite{klebaner2005introduction},
\cite{YW},
there exists a measurable map
 $\bar{\G}^\ep$ such that
 $$X^\ep(\cdot)=\bar{\G}^\ep(\ss\ep W).$$
 Then, by the Girsanov theorem,
 for any $v\in\bar{\mathscr{A}}$,
 $X^{\ep,v}:=\bar{\G}^\ep\Big(\ss\ep W+\ss\ep\lam(\ep)\int_0^\cdot \dot{v}(s)\d s\Big)$ solves
 \begin{equation}\label{eq5.1}
 \begin{split}
 \d (X^{\ep,v}(t)-G(X^{\ep,v}(t-\tau)))&=b(X^{\ep,v}(t),X^{\ep,v}(t-\tau))\d t\\
 &~~+\ss\ep\sigma(X^{\ep,v}(t),X^{\ep,v}(t-\tau))\d W(t)\\
 &~~+\ss\ep\lam(\ep)\sigma(X^{\ep,v}(t),X^{\ep,v}(t-\tau))\dot{v}(t)\d t,~~~~~ t>0.\\
 \end{split}
 \end{equation}
So there exists a measurable map $\G^\ep:\bar{\mathscr{A}}\mapsto\R^n$ such that
 \begin{equation}\label{eq45}
Z^{\ep, v}:=\G^\ep\Big(\ss\ep W+\ss\ep\lam(\ep)\int_0^\cdot \dot{v}(s)\d s\Big)=\ff{X^{\ep,v}-X^0}{\ss\ep\lam(\ep)}.
\end{equation}
As a result,
\begin{equation}\label{eq4.4}
\begin{split}
&\d \Big(Z^{\ep,v}(t)-\ff{G(X^{\ep,v}(t-\tau))-G(X^0(t-\tau))}{\ss\ep\lam(\ep)}\Big)\\
&=\ff{1}{\lam(\ep)}\sigma(X^{\ep,v}(t),
X^{\ep,v}(t-\tau))\d W(t)+\sigma(X^{\ep,v}(t),X^{\ep,v}(t-\tau))\dot{v}(t)\d t\\
&~~~+\ff{b(X^{\ep,v}(t),X^{\ep,v}(t-\tau))
-b(X^0(t),X^{\ep,v}(t-\tau))}{\ss\ep\lam(\ep)}\d t\\
&~~~+\ff{b(X^0(t),X^{\ep,v}(t-\tau))-b(X^0(t),
X^0(t-\tau))}{\ss\ep\lam(\ep)}\d t
\end{split}
\end{equation}
with the initial condition $Z^{\ep,v}(\theta)={\bf0_n}$, for any $\theta\in[-\tau,0].$\\
\begin{lem}\label{lem1.1}
Assume ({\bf A2}) holds, there is a constant~$C>0$~and~$\ep_0\in(0,1)$~such that
\begin{equation}\label{eq11.1}
\begin{split}
\mathbb{E}\big(\sup_{0\le t\le T}|X^{\ep,v}(t)|^p\big)\vee
\big(\sup_{0\le t\le T}|X^{0}(t)|^p\big)\le C,~~p\ge 2,\ep\in(0,\ep_0),
\end{split}
\end{equation}
where~$C$~depends on~$\|\xi\|_\8:=\sup_{-\tau\le\theta\le0}|\xi(\theta)|$.~
\end{lem}
\begin{proof}
Let~$q'=q_1+1$~and Set~$M_\ep(t):=X^{\ep,v}(t)-G(X^{\ep,v}(t-\tau))$.~For any~$p\ge2$,~
\begin{equation}\label{eq1.2}
|X^{\ep,v}(t)|^p\le C(1+|M_\ep(t)|^p+|X^{\ep,v}(t-\tau)|^{pq'})
\end{equation}
\begin{equation}\label{eq11.3}
\begin{split}
\sup_{0\le s\le t}|X^{\ep,v}(s)|^p&
\1 1+\sup_{-\tau\le s\le t-\tau}|X^{\ep,v}|^{pq'}+M_{1,\ep}(t)\\
&~~~~+p\int_0^t|M_\ep(s)|^{p-2}\langle M_\ep(s),b(X^{\ep,v}(s),X^{\ep,v}(s-\tau))\rangle\d s\\
&~~~~+\frac{\ep p(p-1)}{2}\int_0^t|M_\ep(s)|^{p-2}
\|\sigma(X^{\ep,v}(s),X^{\ep,v}(s-\tau))\|_{HS}^2\d s\\
&~~~~+p\ss\ep\lambda(\ep)\int_0^t|M_\ep(s)|^{p-2}\langle M_\ep(s),\sigma(X^{\ep,v}
(s),X^{\ep,v}(s-\tau))\cdot \dot{v}(s)\rangle\d s,
\end{split}
\end{equation}
where
\begin{equation*}
M_{1,\ep}(t):=p\ss\ep\sup_{0\le s\le t}\Big|\int_0^s|M_\ep(z)|^{p-2}\langle M_\ep(z),\sigma(X^{\ep,v}(z),X^{\ep,v}(z-\tau))\d W(z)\rangle\Big|.
\end{equation*}
Firstly, we consider the~$M_{1,\ep}(t)$,~utilizing the B-D-G's inequality, it follows that
\begin{equation*}
\begin{split}
\mathbb{E}M_{1,\ep}(t)&\le4p\ss{2\ep}\mathbb{E}\big(\int_0^t|M_\ep(s)|^{2p-2}
\|\sigma(X^{\ep,v}(s),X^{\ep,v}(s-\tau))\|_{HS}^2\d s\big)^\frac{1}{2}\\
&\le\frac{1}{2}\mathbb{E}\sup_{0\le s\le t}|M_\ep(s)|^p+
16p^2\ep\mathbb{E}\int_0^tM_\ep(s)|^{p-2}
\|\sigma(X^{\ep,v}(s),X^{\ep,v}(s-\tau))\|_{HS}^2\d s.
\end{split}
\end{equation*}
\begin{equation*}
\begin{split}
&p\ss\ep\lambda(\ep)\int_0^t|M_\ep(s)|^{p-2}\langle M_\ep(s),\sigma(X^{\ep,v}
(s),X^{\ep,v}(s-\tau))\cdot \dot{v}(s)\rangle\d s\\
&\le p\ss\ep\lambda(\ep)\Big\{\int_0^t|M_\ep(s)|^{p-2}
\|\sigma(X^{\ep,v}(s),X^{\ep,v}(s-\tau))\|_{HS}^2\d s
+\int_0^t\sup_{0\le r\le s}|M_\ep(r)|^{p}|\dot{v}(s)|^2\d s\Big\}.
\end{split}
\end{equation*}
Combining the above estimates and \eqref{eq1.5},
$16p^2\ep_0\vee p\ep_0\lambda(\ep)\le\frac{p(p-1)}{2},~\ep\in(0,\ep_0)$,~one gets that
\begin{equation*}
\begin{split}
\mathbb{E}\Big(\sup_{0\le s\le t}|X^{\ep,v}(s)|^p\Big)
&\11+\mathbb{E}\Big(\sup_{-\tau\le s\le t-\tau}|X^{\ep,v}(s)|^{pq'}\Big)\\
%+\ss\ep\mathbb{E}\Big(\sup_{0\le s\le t}|X^{\ep,v}(s)|^p\Big)\\
&~~+\mathbb{E}\int_0^t\Big(-a_2|X^{\ep,v}(s)|^p+a_3|X^{\ep,v}(s-\tau)|^p
-a_4|X^{\ep,v}(s)|^q+a_5|X^{\ep,v}(s-\tau)|^q\Big)\d s
\end{split}
\end{equation*}

%Hence, for any~$p\ge2$,~we have
%\begin{equation}\label{eq1.8}
%\begin{split}
%\mathbb{E}\Big(\sup_{0\le s\le t}|X^{\ep,v}(s)|^p\Big)&\11+
%\ss{\ep_0}\mathbb{E}\Big(\sup_{0\le s\le t}|X^{\ep,v}(s)|^p\Big)\\
%&~~+\mathbb{E}
%\Big(\sup_{-\tau\le s\le t-\tau}\big(|X^{\ep,v}(s)|^{pq'}
%+|X^{\ep,v}(s)|^{q}\big)\Big)
%\end{split}
%\end{equation}
Then, for~$\ep\in(0,\ep_0)$, we drive  that
\begin{equation*}
\begin{split}
\mathbb{E}\Big(\sup_{0\le s\le t}|X^{\ep,v}(s)|^p\Big)&\11
+\mathbb{E}
\Big(\sup_{-\tau\le s\le t-\tau}|X^{\ep,v}(s)|^{pq'}\Big).
%+\sup_{-\tau\le s\le t-\tau}|X^{\ep,v}(s)|^{q}\Big).
\end{split}
\end{equation*}
Thus, the assertion \eqref{eq11.1} is established by induction argument.
\end{proof}
For any $h\in S$,  define $\G^0(h)=Z^h,$ where $Z^h$ solves
\eqref{eq5.2}.
\begin{lem}\label{le4.1}
Assume {\rm\textbf{(A1)-(A2)}} hold, and suppose that
$h_n\rightarrow h $ as $n\rightarrow\8$ for any $h_n, h\in
S_N$. Then, as $n\rightarrow\8,$
$$\G^0(h_n)\rightarrow \G^0(h).$$

\end{lem}
\begin{proof}
From the notion of $\G^0$, it suffices to show that
\begin{equation}\label{suo0}
\lim_{\ep\rightarrow 0}\sup_{0\le t\le T} |Z^{h_n}(t)-Z^h(t)|=0
\end{equation}
for any $h_n, h\in S_N$. Whereas, to derive
\eqref{suo0},
according to the Arzel\`{a}-Ascoli theorem %\cite{hanche2010kolmogorov},
(see, e.g., \cite[Theorem 4.9]{KI}),
 ~we need only  show that
\begin{enumerate}
\item[(i)] $\{Z^{h_n}(\cdot)\}_{\ep\in(0, 1)}$ is uniformly bounded,
i.e., $\sup_{\ep\in(0,1)}\sup_{t\in[0,T]}|Z^{h_n}(t)|<\8;$
\item[(ii)] $\{Z^{h_n}(\cdot)\}_{\ep\in(0, 1)}$ is equicontinuous,
 i.e., $\lim_{\delta\rightarrow 0}\sup_{\ep\in(0,1)}|Z^{h_n}(t+\delta)-Z^{h_n}(t)|=0.$
\end{enumerate}
In what follows, we verify that (i) and (ii) hold one-by-one. With
the aid of  \eqref{eq1.5}, \eqref{eq11.1}, and
$X^\ep(\theta)=X^0(\theta)$ for any $\theta\in[-\tau,0]$,  we derive
from H\"older's inequality that for any   $h_n \in S_N$,
 \begin{equation*}
 \begin{split}
 |Z^{h_n}(t)|&\1V(X^0(t-\tau),X^0(t-\tau))|Z^{h_n}(t-\tau)|\\
 &~~+\Big(\int_0^t\|\sigma(X^0(s),X^0(s-\tau))\|_{HS}^2\d s\Big)^{\ff{1}{2}}\Big(\int_0^t|\dot{h}_n(s)|^2\d s\Big)^{\ff{1}{2}}\\
 &~~+\int_0^tV(X^0(s),X^0(s))|Z^{h_n}(s)|\d s+\int_0^{(t-\tau)\vee0}V(X^0(s),X^0(s))|Z^{h_n}(s)|\d s\\
 &\1 1+\sup_{0\le s\le( t-\tau)\vee0}|Z^{h_n}(s)|+\int_0^t|Z^{h_n}(s)|\d s.
 \end{split}
 \end{equation*}
By the Gronwall inequality, one has
 \begin{equation*}
 |Z^{h_n}(t)|\11+\sup_{0\le s\le( t-\tau)\vee0}|Z^{h_n}(s)|.
 \end{equation*}
Then (i) follows from an induction argument.\\
%Taking the similar argument of Lemma \ref{le3.4}, the proof of (i) is completed.\\
% we get that
%\begin{equation*}
% |Z^{h_n}(t)|^p<C, ~~t\in[0,T]
% \end{equation*}
% for some constant $C>0.$ The proof of (i) is completed.
In the sequel, without loss of generality, we assume $\dd\in(0,1).$ From ({\bf A1}), \eqref{eq1.5} and \eqref{eq11.1}, one has  that
\begin{equation*}
\begin{split}
|X^0(t+\delta)-X^0(t)|&\le|G(X^0(t+\delta-\tau))-G(X^0(t-\tau))|+\int_t^{t+\delta}|b(X^0(s),X^0(s-\tau))|\d s\\
&\le V(X^0(t+\delta-\tau),X^0(t-\tau))|X^0(t+\delta-\tau))-X^0(t-\tau))|\\
&~~+\int_t^{t+\delta}\{1+|X^0(s-\tau)|^2+|X^0(s-\tau)|^q\}^{\frac{1}{2}}\d s\\
&\1|X^0(t+\delta-\tau))-X^0(t-\tau))|+\delta.\\
%&\le V(X^0(t+\delta-\tau),X^0(t-\tau)))|X^0(t+\delta-\tau))-X^0(t-\tau))|+\delta
\end{split}
\end{equation*}
Hence, we further get that
\begin{equation}\label{eq77}
 |X^0(t+\delta)-X^0(t)|\1\delta, ~~~t\in[0,T]
 \end{equation}
  by an induction argument and the continuity of the initial value.

With the result of (i), and taking \eqref{eq1.5}, \eqref{eq11.1} into consideration,
for~$p=2$,~we have
\begin{equation*}
\begin{split}
|Z^{h_n}(t+\delta)-Z^{h_n}(t)|
&\1|\nabla_{Z^{h_n}(t+\delta-\tau)}(G(X^0(t+\delta-\tau))-G(X^0(t-\tau)))|\\
&~~+|\nabla_{Z^{h_n}(t+\delta-\tau)-Z^{h_n}(t-\tau)}G(X^0(t-\tau))|\\
&~~+\int_t^{t+\delta}\|\sigma(X^0(s),X^0(s-\tau))\|_{HS}^2\d s
+\int_t^{t+\delta}|\dot{h}_n(s)|^2\d s\\
&~~+\int_t^{t+\delta}V(X^0(s),X^0(s))|Z^{h_n}(s)|\d s+\int_{t-\tau}^{t-\tau+\delta}V(X^0(s),X^0(s))|Z^{h_n}(s)|\d s\\
&\1V(X^0(t+\delta-\tau),X^0(t-\tau))|X^0(t+\delta-\tau)-X^0(t-\tau)|\cdot|Z^{h_n}(t+\delta-\tau)|\\
&~~+V(X^0(t-\tau),X^0(t-\tau))|Z^{h_n}(t+\delta-\tau)-Z^{h_n}(t-\tau)|\\
&~~+\delta+\Big(\int_{t-\tau}^{t-\tau+\delta}|Z^{h_n}(s)|^2\d s\Big)^{\ff{1}{2}}\\
&\1\delta+|Z^{h_n}(t+\delta-\tau)-Z^{h_n}(t-\tau)|,
\end{split}
\end{equation*}
where in the first inequality, we have used ({\bf A3}), the second inequality, we have used \eqref{eq11.7}, \eqref{eq77} and $h_n,h\in S_N$.
As a result, (ii) is established by an induction argument and the continuity of the initial data.\\
Since~$(Z^{h_n}(\cdot))_{n\in\mathbb{N}}$~is pre-compact in~$C([0,T];\mathbb{R}^n)$,~every sequence, which is still denoted by~$(Z^{h_n}(\cdot))_{n\in\mathbb{N}}$,
~has a convergent subsequence. So we conclude that~$Z^h(\cdot)$~be the limit point of~$(Z^{h_n}(\cdot))_{n\in\mathbb{N}}$~from the uniqueness.

%when $t\in[0,\tau]$, we can proceed this argument on $[\tau,2\tau],[2\tau,3\tau]$ etc, and hence obtain (ii) on the entire interval $[0,T]$.
\end{proof}
\begin{lem}\label{le4.3}
 Let $v,v_\ep\in\bar{\mathscr{A}}_N$ such that $v_\ep$ converges weakly to $v$,
 as $\ep\rightarrow 0$.  Then
 $$\G^\ep\Big(\ss\ep W+\ss\ep\lam(\ep)\int_0^\cdot
 \dot{v}_\ep(s)\d s\Big)\Rightarrow \G^0\Big(\int_0^\cdot\dot{v}(s)\d s\Big),$$
 where ``$\Rightarrow$" stands for  convergence in distribution of random variables.
%$$Z^{\ep,v_\ep}\Rightarrow Z^v,~~~~~\ep\rightarrow 0.$$
\end{lem}
\begin{proof}
To begin, we  show that $\{Z^{\ep,v_\ep}\}_{\ep\in(0,1)}$ is tight in $C([0,T];\R^n)$.
By virtue of the Arzel\`{a}-Ascoli theorem (see, e.g., \cite[Theorem 4.11]{KI}), it is sufficient to verify that
\begin{enumerate}
\item[(i)] $\sup_{\ep\in(0,1)}\E|Z^{\ep,v_\ep}(t)|^\gamma<\8;$
\item[(ii)] $\sup_{\ep\in(0,1)}\E|Z^{\ep,v_\ep}(t)-Z^{\ep,v_\ep}(s)|^\alpha\le C_T|t-s|^{1+\beta};~~~~~0\le s,t\le T,~|t-s|<1~$
\end{enumerate}
for some positive constants $\alpha, \beta,\gamma$ and $C_T.$\\
By the chain rule and the Taylor expansion, in addition to \eqref{eq1.5},
 \eqref{eq1.6}, \eqref{eq1.7}, and \eqref{eq11.7} for $v,v_\ep\in\bar{\mathscr{A}}_N$,
 we derive  that
\begin{equation*}
\begin{split}
\E|Z^{\ep,v_\ep}(t)|^p&\1\E|\nabla_{Z^{\ep,v_\ep}(t-\tau)}G(X^0(t-\tau)+u^\ep(t-\tau))|^p\\
&~~+\lam^{-p}(\ep)\mathbb{E}\Big(\int_0^t
\|\sigma(X^{\ep,v_\ep}(s),X^{\ep,v_\ep}(s-\tau))\|_{HS}^p\d s\\
&~~+
\E\Big(\int_0^t|\dot{v}_\ep(s)|^2\d s\Big)^{\ff{p}{2}}\Big(\int_0^t\|\sigma(X^{\ep,v_\ep}(s),X^{\ep,v_\ep}(s-\tau))\|_{HS}^2\d s\Big)^{\ff{p}{2}}\\
&~~+\int_0^t\E|\nabla_{Z^{\ep,v_\ep}(s)}^{(1)}b(X^0(s)+u^\ep(s),X^{\ep,v_\ep}(s-\tau))|^p\d s\\
&~~+\int_0^t\E|\nabla_{Z^{\ep,v_\ep}(s-\tau)}^{(2)}b(X^0(s),X^0(s-\tau)+u^\ep(s-\tau))|^p\d s\\
&\11+\lam^{-p}(\ep)+\sup_{0\le s\le( t-\tau)\vee0}\Big(\E|Z^{\ep,v_\ep}(s)|^{2p}\Big)^{\ff{1}{2}}+\int_0^t\E|Z^{\ep,v_\ep}(s)|^p\d s\\
&~~+\int_0^t(\E|Z^{\ep,v_\ep}(s-\tau)|^{2p})^{\ff{1}{2}}\d s, ~~~~t\in[0,T], ~~~p\ge 2.
\end{split}
\end{equation*}
In the last step, we utilize the \eqref{eq11.1} and the H\"older inequality, Then taking  Gronwall's inequality into consideration, one has that
\begin{equation}\label{eq11}
\begin{split}
\E|Z^{\ep,v_\ep}(t)|^p\11+\lam^{-p}(\ep)+\sup_{0\le s\le( t-\tau)\vee0}\Big(\E|Z^{\ep,v_\ep}(s)|^{2p}\Big)^{\ff{1}{2}}+\int_0^t(\E|Z^{\ep,v_\ep}(s-\tau)|^{2p})^{\ff{1}{2}}\d s.
\end{split}
\end{equation}
Hereinafter, by mimicking the argument of Theorem \ref{th3.2}, we get that
% Let
%$$\gamma_i=([T/\tau]+2-i)2\gamma ^{[T/\tau]+1-i},~~~~~i=1,2,\cdots,[T/\tau]+1.$$
%It is easy to see that $\gamma_i\ge 2$ such that
%$$2\gamma_{i+1}<\gamma_i~~~~~\mbox{and}~~~~~\gamma_{[T/\tau]+1}=\gamma,~~~i=1,2,\cdots,[T/\tau].$$
%Noting that $Z^{\ep,v_\ep}(t)=0$ for $t\in[-\tau,0]$, by \eqref{eq11} we get
%\begin{equation*}
%\E|Z^{\ep,v_\ep}(t)|^{\gamma_1}\1\lam^{-\gamma_1}(\ep),~~~~t\in[0,\tau].
%\end{equation*}
%This, together with \eqref{eq11} and the H\"{o}lder inequality, yields that
%\begin{equation*}
%\begin{split}
%\E|Z^{\ep,v_\ep}(t)|^{\gamma_2}\1&\lam^{-\gamma_2}(\ep)+(\E|Z^{\ep,v_\ep}(t-\tau)|^{\gamma_1})^{\ff{\gamma_2}{\gamma_1}}\\
%&~~+\int_0^t(\E|Z^{\ep,v_\ep}(s-\tau)|^{\gamma_1})^{\ff{\gamma_2}{\gamma_1}}\d s\\
%&\1\lam^{-\gamma_2}(\ep),~~~t\in[0,2\tau].
%\end{split}
%\end{equation*}
%Following the previous procedures gives that

\begin{equation*}
\E|Z^{\ep,v_\ep}(t)|^p\11+\lam^{-p}(\ep),~~~~t\in[0,T].
\end{equation*}
 Hence (i) holds with $\gamma=p$. \\
In the sequel, note that  \eqref{eq1.6}, \eqref{eq1.7}, \eqref{eq11.7} and \eqref{eq11.1}. Then utilizing the Taylor expansion fields that
\begin{equation*}
\begin{split}
&\E|Z^{\ep,v_\ep}(t)-Z^{\ep,v_\ep}(s)|^\alpha\\
&\1\E|\nabla_{Z^{\ep,v_\ep}(t-\tau)}G(X^0(t-\tau)+u^\ep(t-\tau))
-\nabla_{Z^{\ep,v_\ep}(t-\tau)}G(X^0(s-\tau)+u^\ep(s-\tau))|^\alpha\\
&~~+\E|\nabla_{Z^{\ep,v_\ep}(t-\tau)-Z^{\ep,v_\ep}(s-\tau)}G(X^0(s-\tau)
+u^\ep(s-\tau))|^\alpha\\
&~~+\ff{|t-s|^{\ff{\alpha-2}{2}}}{\lam^\alpha(\ep)}
\int_s^t\E\|\sigma(X^{\ep,v_\ep}(r),X^{\ep,v_\ep}(r-\tau))\|_{HS}^\alpha\d r\\
&~~+|t-s|^{\frac{\alpha}{\alpha-1}}
\int_s^t\E|\nabla_{Z^{\ep,v_\ep}(r)}^{(1)}b(X^0(r)+u^\ep(r),X^{\ep,v_\ep}(r-\tau))|^\alpha\d r\\
&~~+|t-s|^{\frac{\alpha}{\alpha-1}}
\int_s^t\E|\nabla_{Z^{\ep,v_\ep}(r-\tau)}^{(2)}b(X^0(r),X^0(r-\tau)+u^\ep(r-\tau))|^\alpha\d r\\
&\1|t-s|^{\frac{2\alpha-1}{\alpha-1}}+\ff{|t-s|^{\ff{\alpha}{2}}}{\lam^\alpha(\ep)}
+(\E|Z^{\ep,v_\ep}(t-\tau)-Z^{\ep,v_\ep}(s-\tau)|^{2\alpha})^\ff{1}{2},~~~0\le s,t\le T.
\end{split}
\end{equation*}
Next, taking the induction argument has that
$$\E|Z^{\ep,v_\ep}(t)-Z^{\ep,v_\ep}(s)|^\alpha\1|t-s|^{\frac{2\alpha-1}{\alpha-1}}
+\ff{|t-s|^{\ff{\alpha}{2}}}{\lam^\alpha(\ep)}.$$
Therefore, (ii) holds with $\alpha=2(1+\beta).$
 Thus $\{Z^{\ep,v_\ep}\}_{\ep\in(0,1)}$ is tight in $C([0,T];\R^n)$.\\
In the sequel, it suffices to show that $Z^v$ is the unique limit point of $\{Z^{\ep,v_\ep}\}_{\ep\in(0,1)}.$ Let
$$M^\ep(t)=\ff{1}{\lam(\ep)}\int_0^t\sigma(X^{\ep,v_\ep}(s),X^{\ep,v_\ep}(s-\tau))\d W(s).$$
Since $\{Z^{\ep,v_\ep}\}_{\ep\in(0,1)}$ is tight in $C([0,T];\R^n)$, we can choose a subsequence of $(Z^{\ep,v_\ep},v_\ep,M^\ep)$
convergent weakly to $(Y,v,0)$  as $\ep\rightarrow 0.$ Without loss of generality,
 by the Skorokhod representation theorem \cite{SK},
there exists a probability space $(\bar{\Omega},\bar{\mathcal{F}},\{\bar{\mathcal{F}_t}\}_{t\ge 0 },\bar{\P})$,
 on this basis, an  Brownian motion $\bar{W}$ and
  a family of $\bar{\mathcal{F}}$-predictable process $\{\bar{v}_\ep;\ep>0\},\bar{v}\in\bar{\mathscr{A}}$
  taking values on $\bar{\mathscr{A}}_N,\bar{\P}$-a.s., such that the joint
  law of $(v_\ep,v,W)$ under $\P$ coincides with that of $(\bar{v}_\ep,\bar{v},\bar{W})$ under $\bar{\P}$ and
  $$\lim_{\ep\rightarrow 0}\int_0^T<\bar{v}_\ep-\bar{v},g>\d s=0,~
  \forall g\in\bar{\mathscr{A}},\bar{\P}-a.s.$$
%$$\lim_{\ep\rightarrow 0}<\bar{v}_\ep-\bar{v},g>_{\mathcal{S}}=0,~~\forall g\in\mathcal{A},\bar{\P}-a.s.$$
Without confusion, we drop the bars in the notation.
Thus, we may assume
$$(Z^{\ep,v_\ep},v_\ep,M^\ep)\rightarrow (Y,v,0),~~~\bar{\P}-a.s.$$
Taking $\ep\rightarrow 0$ on both sides of \eqref{eq4.4}, we infer that $Y$ also satisfies \eqref{eq5.2}. Thus the desired assertion follows from the uniqueness.
\end{proof}

\noindent\textbf{The proof of Theorem \ref{th5.1}}
\begin{proof}
With Lemmas \ref{le4.1} and \ref{le4.3} in hand and by taking
\cite[Theorem 4.4]{WV} into account,
 the proof of Theorem \ref{th5.1} can be completed.  \end{proof}
\section{Examples}\label{sec4}
Theorem \ref{th5.1} covers many highly nonlinear SDDEs. Let us discuss two examples at the end of this section.
\begin{exa}\label{ex1}
Consider a one-dimensional SDDE
\begin{equation}\label{eqq1}
\d x(t)=[x^2(t-\tau)-2x(t)-x^3(t)]\d t+\ss\ep x^2(t-\tau)\d B(t),
\end{equation}
where~$B(t)$~is a one dimensional Brownian motion.\\
For simplicity, let~$y=x(t-\tau),x=x(t)$,~
\begin{equation*}
\begin{split}
2x(-x^3-2x+y^2)+y^4\le -2x^4-3x^2+y^2+y^4.
\end{split}
\end{equation*}
where~$p=2,~q=4$~and~$a_1=0, a_2=2,a_3=1,a_4=3,a_5=1$,~
\begin{equation*}
\|\nabla^{(1)}b(x_1,\cdot)-\nabla^{(1)}b(x_2,\cdot)\|\le3(x_1+x_2)|x_1-x_2|,
\end{equation*}
\begin{equation*}
\|\nabla^{(2)}b(\cdot,y_1)-\nabla^{(2)}b(\cdot,y_2)\| \le2|y_1-y_2|.
\end{equation*}
Assumptions {\rm\textbf{(A1)-(A3)}} are therefore satisfied, then the conclusion of theorem \ref{th5.1} is established.
\end{exa}
\begin{exa}\label{ex2}
We consider another equation
\begin{equation*}
\d x(t)=-x^3(t)-2x(t)+x(t-\tau)\d t+\ss\ep x^2(t)\d B(t)
\end{equation*}
For simplicity, we set~$y=x(t-\tau),x=x(t)$,~
\begin{equation*}
2x(-x^3-2x+y)+x^4\le-2x^4-3x^2+y^2,
\end{equation*}
where~$p=2,~q=4$~and~$a_1=0, a_2=3,a_3=1,a_4=2,a_5=0$,~
\begin{equation*}
\|\nabla^{(1)}b(x_1,\cdot)-\nabla^{(1)}b(x_2,\cdot)\|\le3(x_1+x_2)|x_1-x_2|,
\end{equation*}
\begin{equation*}
\|\nabla^{(2)}b(\cdot,y_1)-\nabla^{(2)}b(\cdot,y_2)\|=0
\end{equation*}
Assumptions {\rm\textbf{(A1)-(A3)}} are therefore satisfied, then the conclusion of theorem \ref{th5.1} is established.
\end{exa}

\section*{Acknowledgements}
The authors are grateful to the anonymous referees for their valuable comments and corrections.

%\bibliographystyle{plain}
%\bibliography{reference}

\end{document}